\documentclass[12pt,a4paper]{article}
\usepackage{indentfirst}
\usepackage{amsmath,amssymb,amsfonts,amsthm,graphics}
\usepackage{makeidx}
\usepackage{color}

\newtheorem{theorem}{Theorem}

\newtheorem{proposition}{Proposition}
\newtheorem{lemma}{Lemma}

\newtheorem{corollary}{Corollary}

\newtheorem{remark}{Remark}
\newtheorem{observation}{Observation}
\newtheorem{conjecture}{Conjecture}
\begin{document}
\title{\Large\bf Note on the oriented diameter\\ of graphs with diameter $3$\footnote{Supported by NSFC
No.11071130.}}
\author{\small Hengzhe Li, Xueliang Li, Yuefang Sun, Jun Yue
\\
\small Center for Combinatorics and LPMC-TJKLC
\\
\small Nankai University, Tianjin 300071, China
\\
\small lhz2010@mail.nankai.edu.cn; lxl@nankai.edu.cn\\
\small bruceseun@gmail.com; yuejun06@126.com}
\date{}
\maketitle
\begin{abstract}
In 1978, Chv\'atal and Thomassen showed that every bridgeless graph
with diameter $2$ has an orientation with diameter at most $6$. They
also gave general bounds on the smallest value $f(d)$ such that
every bridgeless graph $G$ with diameter $d$ has an orientation with
diameter at most $f(d)$. For $d=3$, they proved that $8\leq f(d)\leq
24$. Until recently, Kwok, Liu and West improved the above bounds by
proving $9\leq f(3)\leq 11$ in [P.K. Kwok, Q. Liu and D.B. West,
Oriented diameter of graphs with diameter $3$, J. Combin. Theory
Ser.B 100(2010), 265-274]. In this paper, we determine the oriented
diameter among the bridgeless graphs with diameter 3 that have
minimum number of edges.

\end{abstract}

\section{Introduction}

All graphs in this paper are finite and simple. A $graph\ G$ is an
ordered pair $(V(G),E(G))$ consisting of a set $V(G)$ of $vertices$
and a set $E(G)$ of $edges$. We refer to book \cite{bondy} for graph
theoretical notation and terminology not given here. An orientation
of a graph $G$ is a $digraph$ obtained from $G$ by replacing each
edge by just one of the two possible arcs with the same ends. We
occasionally use the symbol $\overrightarrow{G}$ to specify an
orientation of $G$ (even though a graph generally has many
orientations). An orientation of a simple graph is referred to as an
$oriented\ graph$. The diameter of a graph $G$ (digraph $D$) is
$\max \{d(u,v)\ |\ u,v\in V(G)\}$ ($\max \{d(u,v)\ |\ u,v\in
V(D)\}$), denoted by $diam(G)\ (diam(D))$. The oriented diameter of
a bridgeless graph is $\min\{ diam(\overrightarrow{G})\ |
\overrightarrow{G}\ is\ an\ orientation\ of\ G\}$.

In 1939, Robbins solved the One-Way Street Problem and proved that a
graph $G$ admits a strongly connected orientation if and only if $G$
is bridgeless, that is, does not have any cut-edge. Naturally, one
hopes that the oriented diameter of a bridgeless graph is as small
as possible. Bondy and Murty suggested to study the quantitative
variations on Robbins' theorem. In particular, they conjectured that
there exists a function $f$ such that every bridgeless graph with
diameter $d$ admits an orientation of diameter at most $f(d)$.

In 1978, Chv\'atal and Thomassen \cite{chv} gave general bounds
$\frac{1}{2}d^2+d \leq f(d)\leq 2d^2+2d$ for $d\geq 2$. These bounds
have not got improved in the past several decades. They also proved
that $f(1)=3$ and $f(2)=6$. For $d=3$, the general result reduces to
$8\leq f(3)\leq 24$. Until recently, Kwok, Liu and West improved the
above bounds by proving $9\leq f(3)\leq 11$ in \cite{kwo}.

In \cite{chv,kwo}, they got the bounds by considering the distances
from two special adjacent vertices to the other vertices, the
distances from the other vertices to the two special vertices, and
the triangle inequality of distances. In this paper, we will
directly consider the distance from one vertex to another in
bridgeless graphs with diameter $3$ that have minimum number of
edges.

Chv\'atal and Thomassen \cite{chv} showed that determining whether
an arbitrary graph may be oriented so that its diameter is at most 2
is NP-complete. Bounds of oriented diameter of graphs have also been
studied in terms of other parameters, for example, radius,
dominating number \cite{chv,fom,sol}, etc. Some classes of graphs
have also been studied in \cite{fom, koh1,koh2,kon,mcc}.

This paper is organized as follows: in Section $2$, we give some
known results and introduce a class of extremal bridgeless graphs
with diameter $3$; in Section $3$, we present a vertex set partition
of this graph class, and specify an orientation under this
partition; in Section $4$, we show the above orientation is optimal.

\section{Preliminaries}

In \cite{kwo}, Kwok, Liu and West gave the following proposition by
constructing the graph of Figure $1a$.

\begin{figure}[h,t,b,p]
\begin{center}
\scalebox{0.9}[0.9]{\includegraphics{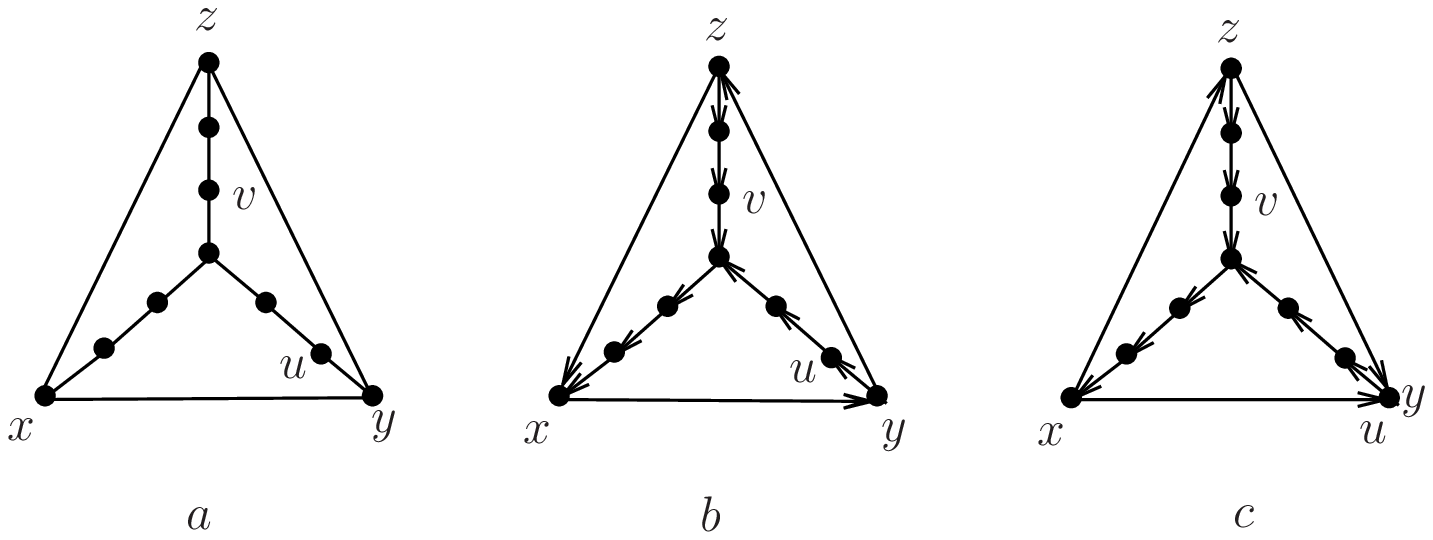}}

Figure 1. A graph of diameter $3$ and its two\\ oriented graphs with
diameter $9$.
\end{center}
\end{figure}

\begin{proposition}\cite{kwo}
$$f(3)\geq 9.$$
\end{proposition}

Clearly, the distance from $u$ to $v$ is $9$ in Figures $1b$ and
$1c$. (each oriented graph is isomorphic to one of Figures $1a$ and
$1b$).

Let $S$ and $S'$ be two disjoint vertex sets. We use $E[S,S']$ to
denote the set of edges having one end in each one of $S$ and $S'$.

\begin{lemma}\cite{kwo}
In a graph $H$, let $S$ and $S'$ be disjoint vertex sets such that
$S'\subseteq N_{H}(S)$, and let $F$ denote the graph $H[S']\cup
(S\cup S',E[S', S])$. If the induced subgraph $H[S']$ is connected
and nontrivial, then there is an orientation $\overrightarrow{F}$ of
$F$ such that $d_{\overrightarrow{F}}(S,w)\leq 2$ and
$d_{\overrightarrow{F}}(w,S)\leq 2$ for every $w\in S'$.
\end{lemma}

\begin{remark}
In fact, the condition ``the induced subgraph $H[S']$ is connected
and nontrivial'' can be replaced by `` the induced subgraph $H[S']$
does not have trivial components'' by the proof of Lemma~$1$.
\end{remark}

Now, we introduce a class of extremal graph. Denote by $G(n, k,
\lambda, s)$ the class of graphs with $n$ vertices and diameter at
most $k$ which have the property that by deleting any $s$ or fewer
edges the resulting subgraphs have diameters at most $\lambda >k$.
Furthermore, denote by Min$G(n, k, \lambda, s)$ the subclass (of
$G(n, k, \lambda, s)$) of graphs with minimum number of edges, and
denote by $M(n,k,\lambda,s)$ the minimum possible number of edges.

In \cite{cac2}, Caccetta gave the following observation and lemma.

\begin{figure}[h,t,b]
\begin{center}
\scalebox{0.7}[0.7]{\includegraphics{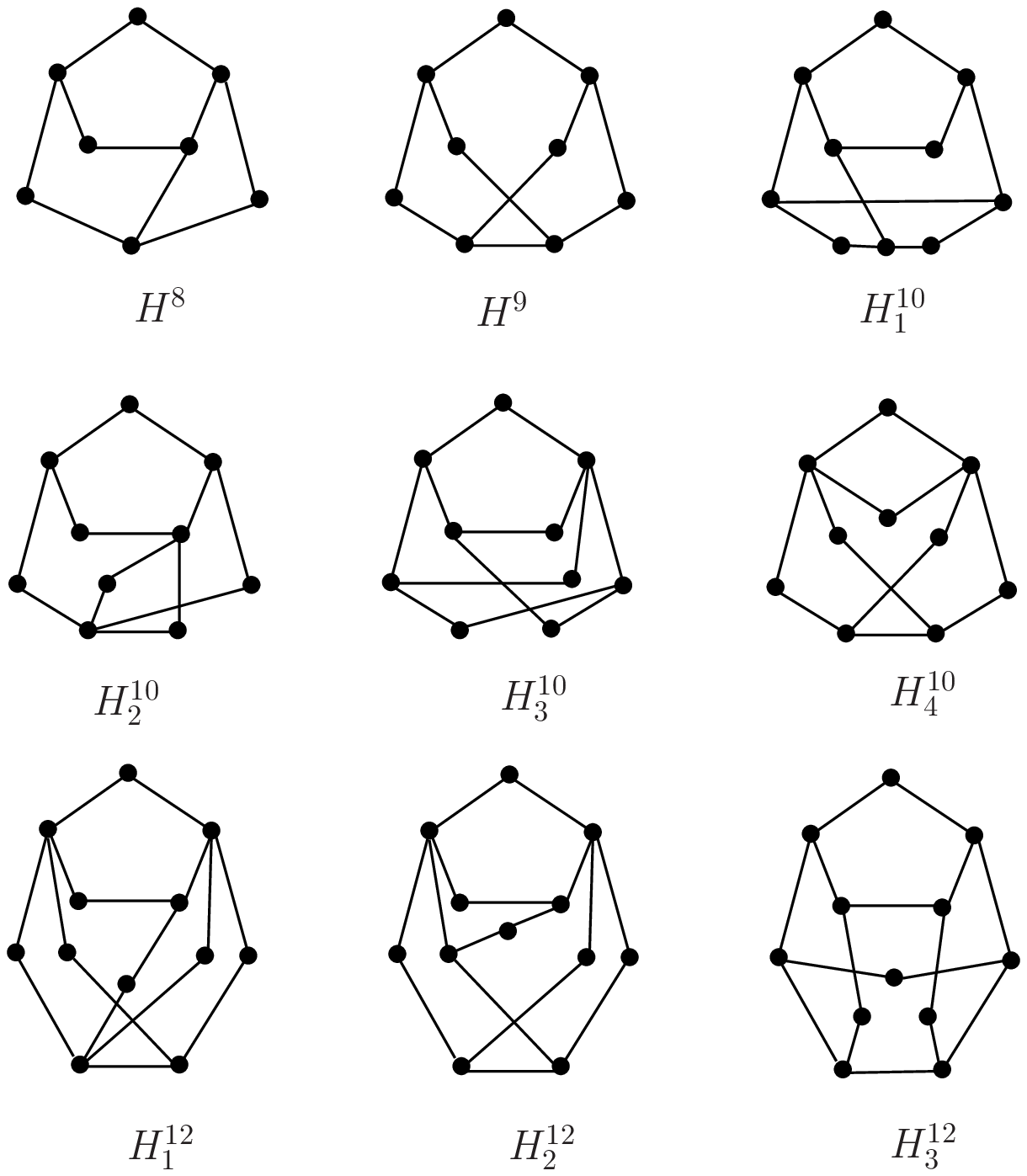}}

Figure 2. The graphs $H^8,H^9,H^{10}_j (j=1,2,3,4)$ and
$H^{12}_j(j=1,2,3).$
\end{center}
\end{figure}

\begin{observation}\cite{cac2}

$(1)$ A $G(n,k, \lambda,s)$ graph is also a $G(n,k', \lambda',s')$
graph, whenever $k'\geq k, \lambda'\geq \lambda$ and $s'\leq s$.
Consequently, the function $M(n, k, \lambda, s)$ is monotonic
non-decreasing in $s$, and monotonic non-increasing in $k$ and
$\lambda$.

$(2)$ In a $G(n,k,\lambda,s)$ there will be at least $s+1$ edge
disjoint paths of length $\leq \lambda$ between any two vertices, at
least one of which has length $\leq k$.

$(3)$ The degree of every vertex of $G$ is at least $s+1$, that is,
$\delta (G)\geq s+1$.

$(4)$ If $\delta (G)= s+1$, then every vertex of $G$ which is not
adjacent to $x$ with degree $\delta (G)$ must be connected to each
of the $s+1$ vertices adjacent to $x$ by a path of length $\leq
\lambda -1$ (from $(2)$).
\end{observation}

\begin{lemma}\cite{cac2}
Let $G\in \mathrm{Min} G(n, 3, \lambda, 1)$, where $\lambda\geq 4$.
Then $G$ possesses two adjacent vertices of degree $2$ for every
$n\geq 5$ except possibly $n=8,9,10$ and $12$. Furthermore, if $G$
does not possess two adjacent vertices of degree $2$, then the only
possible structures are the graphs $H^8,H^9,H^{10}_j (j=1,2,3,4)$
and $H^{12}_j(j=1,2,3)$ in Figure $2$.
\end{lemma}

There are many interesting results on Min$G(n, k, \lambda, s)$. We
refer the readers to \cite{bol1,bol2,bol3,cac1,cac2} for some more
results or details.

\section{Vertex set partition and orientation}
In this section, let $G\in \mathrm{Min} G(n, 3, \lambda, 1)$, where
$n\geq 5$ and $\lambda\geq 4$, and let $u$ and $v$ be two adjacent
vertices of degree $2$ in $G$. Suppose $u(v)$ is adjacent to $x(y)$.
Let $X,Y$ and $Z$ denote the sets $N(x)\setminus N(y),\
N(y)\setminus N(x)$ and $N(x)\cap N(y)$, respectively.

Let $A=X\cup Y\cup Z\cup \{u,v,x,y\}$. For $s\in V(G)\setminus A$,
clearly $d_{G}(s,u)=d_{G}(s,v)=3$ since $G\in G(n,3,\lambda,1)$,
that is $N(s)\cap N(x)\neq\emptyset$ and $N(s)\cap
N(y)\neq\emptyset$. We partition this set based on the distribution
of the neighbors of $s$.
\begin{align*}
& W=(N(X)\cap N(Y))\setminus A;\\
& I=(N(X)\cap N(Z))\setminus A;\\
& K=(N(Y)\cap N(Z))\setminus A;\\
& J=V(G)\setminus (W\cup I\cup K\cup A).
\end{align*}
See Figure $3$ for details.

At this point, we further partition $X$ and $Y$ as follows:
\begin{align*}
& X_1=\{x\in X\ |\ x\ \mathrm{has\ neighbors\ in}\ Y\cup Z\cup I\cup
W\}.\\
& X_2=\{x\in X\setminus X_1\ |\ x\ \mathrm{is\ an\ isolated\ vertex\ in}\ G[X\setminus X_1]\},\\
& X_3=X\setminus (X_1\cup X_2),\\
& Y_1=\{y\in Y\ |\ y\ \mathrm{has\ neighbors\ in}\ X\cup Z\cup K\cup W\},\\
& Y_2=\{y\in Y\setminus Y_1\ |\ y\ \mathrm{is\ an\ isolated\ vertex\ in}\ G[Y\setminus Y_1]\},\\
& Y_3=Y\setminus (Y_1\cup Y_2).
\end{align*}

Note that, in Figure~$3$, if $s$ and $t$ lie in distinct ellipses
and there exists no edge joining the two ellipses, then $s$ and $t$
are nonadjacent in $G$. Generally, the edges drawn in Figure~$3$ do
not indicate complete bipartite subgraph.

\begin{figure}[h,t,b]
\begin{center}
\scalebox{0.7}[0.7]{\includegraphics{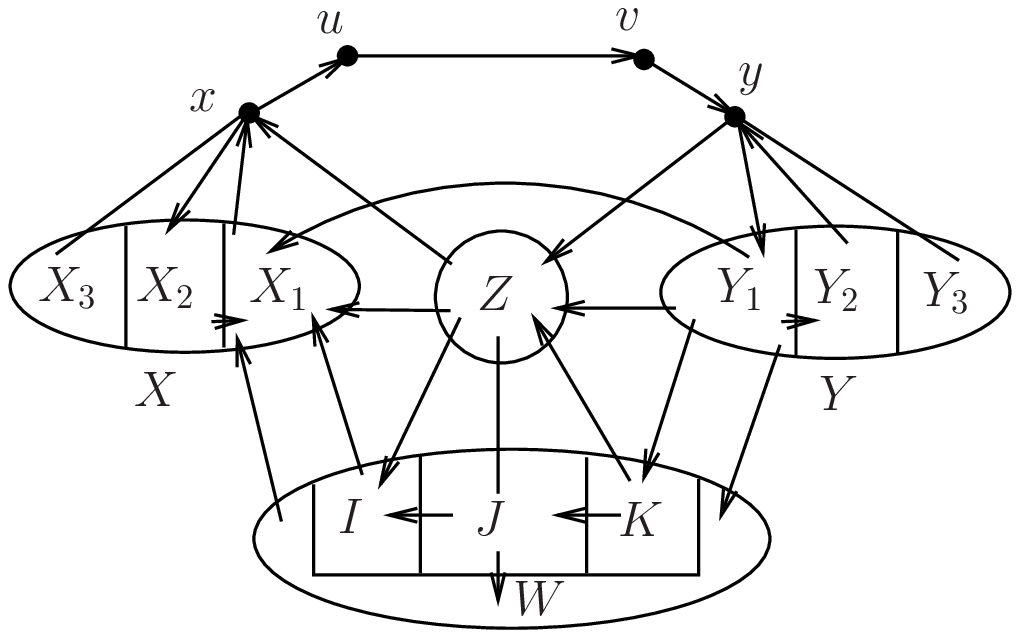}}

Figure 3. An optimal orientation of $G$.
\end{center}
\end{figure}

Given this partition, now we can give an orientation of some edges
of $G$, as shown in Figure $3$.  For the vertex sets (or vertices)
$S$ and $T$, we use the notation $S\rightarrow T$ to mean that all
edges with endpoints in $S$ and $T$ are oriented from $S$ to $T$.
Thus for every sequence below, all edges with ends in two successive
sets are oriented from the first set to the second.
\begin{align*}
& u\rightarrow v\rightarrow y \rightarrow Y_1\rightarrow W\rightarrow X_1\rightarrow x\rightarrow u,\\
& y\rightarrow Z\rightarrow x,\ Y_1\rightarrow Z\rightarrow X_1,\\
& Y_1\rightarrow K\rightarrow Z,\ Z\rightarrow I\rightarrow X_1,\\
& K\rightarrow J\rightarrow I,\ J\rightarrow W,
\end{align*}
\begin{align*}
& x\rightarrow X_2\rightarrow X_1,\ Y_1\rightarrow Y_2\rightarrow y.
\end{align*}

We further partition $J$ as follows:
\begin{align*}
& J_1= J\cap N(K),\ J_2= (J\cap N(I))\setminus J_1,\ J_3= (J\cap N(W))\setminus (J_1\cup J_2),\\
& J_4=J\setminus (J_1\cup J_2\cup J_3)=\{s\in J\ |\ s\ \mathrm{ has\
no\ neighbor\ in}\ I\cup K\cup W\},\\
& J_{4,1}=\{s\in J_4\ |\ N(s)\in Z\},\\
& J_{4,2}=J_4\setminus J_{4,1}=\{\ s\in J_4\ | \ s\in N(Z)\cup
N(J)\}.
\end{align*}

We further orient the edges of $G$ as follows:

\centerline{$J_1\rightarrow Z,\ Z\rightarrow J_2,\ Z\rightarrow
J_3$.}

Since $G\in \mathrm{Min} G(n, 3, \lambda, 1)$, then $\delta(G)\geq
2$ by Observation~1. Thus, for any $s\in J_{4,1}$, we oriented one
of $E[s,Z]$ from $s$ to $Z$, and the others from $Z$ to $s$. By the
above definition, we know that $G[J_{4,2}]$ does not have trivial
components.

Clearly, by Lemma~$1$ and Remark~$1$, there exist orientations
$F_1,F_2$ and $F_3$ of graphs $G[J_{4,2}]\cup (J_{4,2}\cup
Z,E[J_{4,2},Z])$, $G[X_3]\cup (X_3\cup\{x\},E[X_3,\{x\}])$ and
$(G[Y_3]\cup (Y_3\cup\{y\},E[Y_3,\{y\}])$, respectively, such that
for any $s_1\in J_{4,2}$, $s_2\in X_3$ and $s_3\in Y_3$, we have
$d_{F_1}(s_1,Z)\leq 2$, $d_{F_1}(Z,s_1)\leq 2$, $d_{F_2}(s_2,x)\leq
2$, $d_{F_2}(x,s_2)\leq 2$, $d_{F_3}(s_3,y)\leq 2$ and
$d_{F_3}(y,s_3)\leq 2$. The other edges can be oriented arbitrarily.

\section{Determine oriented diameter}

In this section, we show that every $G\in \mathrm{Min} G(n, 3,
\lambda, 1)$ has an oriented diameter at most $9$.

For the orientation in Section $3$, we have the following three
observations.

\begin{observation}
$(1)$ For every $x\in X_2$, $N(x)\cap X_1\neq\emptyset$; $(2)$ For
every $y\in Y_2$, $N(y)\cap Y_1\neq\emptyset$.
\end{observation}

It is easy to see that this observation holds by the definition of
the vertex set partition.

\begin{observation}
$(1)$ For every $s\in X_1$, there exists a path with length at most
$3$ from $y$ to $s$ in the oriented graph $\overrightarrow{G}$;
$(2)$ For every $s\in Y_1$, there exists a path with length at most
$3$ from $s$ to $x$ in the oriented graph $\overrightarrow{G}$.
\end{observation}

\begin{proof}
Since the proof methods are similar, we only show $(1)$. For any
$s\in X_1$, by the definition of set $X_1$, we know that $s$ has
neighbors in $Y\cup Z\cup I\cup W$. Let $t$ be a neighbor of $s$ in
$Y\cup Z\cup I\cup W$. If $t\in Y$, then $t\in Y_1$ by the
definition of set $Y_1$. Thus $y,t,s$ is a path with length $2$ from
$y$ to $t$. If $t\in Z$, then $y,t,s$ is a path with length $2$ from
$y$ to $t$. If $t\in I$, then $y,z,t,s$ is a path with length $3$
from $y$ to $t$, where $z$ is a neighbor of $t$ in $Z$ (Note that
such a vertex $z$ must exist by the definition of set $I$. Otherwise
$t\in W$, then $y,y',t,s$ is a path with length $3$ from $y$ to $t$,
where $y'$ is a neighbor of $t$ in $Y_1$.
\end{proof}

\begin{observation}
If $Z\neq \emptyset$, then $(1)$ for every $s\in X_2\cup X_3$, there
exists a path with length at most $4$ from $y$ to $s$ in the
oriented graph $\overrightarrow{G}$; $(2)$ for every $s\in Y_2 \cup
Y_3$, there exists a path with length at most $4$ from $s$ to $x$ in
the oriented graph $\overrightarrow{G}$.
\end{observation}
\begin{proof}
Since the proof methods are similar, we only show $(1)$. Since
$Z\neq \emptyset$, we pick any $z\in Z$. Then $y,z,x$ is a path with
length $2$ from $y$ to $x$. If $s\in X_2$, then $y,z,x,s$ is a path
with length $3$ from $y$ to $s$. Otherwise, that is, $s\in X_3$,
there exists a path $P$ with length at most $2$ from $x$ to $s$ by
Lemma~$1$ and Remark~$1$. Thus, a path with length at most $4$ from
$y$ to $x$ is obtained from the path $y,z,x$ and $P$.
\end{proof}

\begin{lemma}
Let $G\in \mathrm{Min} G(n, 3, \lambda, 1)$, where $n\geq 5$ and
$\lambda\geq 4$. If $G$ possesses two adjacent vertices of degree
$2$, then the oriented diameter of $G$ is at most $9$.
\end{lemma}
\begin{proof}

We show that the oriented graph $\overrightarrow{G}$ in Section~$3$
has a diameter at most $9$. Consider the following two cases.

{\bf Case $1.$} $Z\neq \emptyset$.

We only show that for any $(s,t)\in (J,V(G)\setminus J),
(V(G)\setminus J,J)$ and $(J,J)$, the distance from $s$ to $t$ is at
most $9$, since it is easy to check the remaining cases hold.
\begin{figure}[htb]
{\scriptsize
\begin{center}
\begin{tabular}{|p{0.5cm}|p{2.1cm}|p{2cm}|p{2cm}|p{2cm}|p{2.5cm}|}
\hline & $J_1$ & $J_2$ & $J_3$ & $J_{4,1}$  & $J_{4,2}$\\\hline

$x$ &$x,u,v,y,Y_1,K,J_1$& $x,u,v,y,Z,J_2$ & $x,u,v,y,Z,J_3$ &
$x,u,v,y,Z,J_{4,1}$ & $x,u,v,y,Z$ and Lemma $1$
\\\hline

$u$ &$u,v,y,Y_1,K,J_1$& $u,v,y,Z,J_2$ & $u,v,y,Z,J_3$ &
$u,v,y,Z,J_{4,1}$ & $u,v,y,Z$ and Lemma~$1$
\\\hline

$v$ &$v,y,Y_1,K,J_1$& $v,y,Z,J_2$ & $v,y,Z,J_3$ & $v,y,Z,J_{4,1}$ &
$v,y,Z$ and Lemma~$1$
\\\hline

$y$ &$y,Y_1,K,J_1$& $y,Z,J_2$ & $y,Z,J_3$ & $y,Z,J_{4,1}$ & $y,Z$
and Lemma $1$
\\\hline

$Y_1$ & Obs~$3$ and $x,$ $u,$ $v,y,Y_1,K,J_1$& Obs~$3$ and $x,$
$u$,$v$,$y$,$Z$,$J_2$ &  Obs~$3$ and $x$, $u$,$v$,$y$,$Z$,$J_3$ &
Obs~$3$ and $x$,$u$,$v$,$y$,$Z$,$J_{4,1}$ & Obs~$3$,
$x$,$u$,$v$,$y$,$Z$ and Lemma $1$
\\\hline

$Y_2$ &$Y_2$,$y$,$Y_1$,$K$,$J_1$& $Y_2$,$y,Z,J_2$ & $Y_2$,$y,Z,J_3$
& $Y_2$,$y,Z,J_{4,1}$ & $Y_2$,$y,Z$~and~Lemma~$1$
\\\hline

$Y_3$ & Lemma~$1$ and $y,Y_1,K,J_1$& Lemma~$1$ and $y,Z,J_2$ &
Lemma~$1$ and $y,Z,J_3$ & Lemma~$1$ and $y,Z,J_{4,1}$ & Lemma~$1$,
$y,Z$ and Lemma~$1$
\\\hline

$K$ &$K,Z,x,u,v,y,Y_1,$ $K,J_1$ & $K,Z,x,u,v,y,Z,$ $J_2$ &
$K,Z,x,u,v,y,Z,$ $J_3$ & $K,Z,x,u,v,y,Z,$ $J_{4,1}$ &
$K,Z,x,u,v,y,Z$ and Lemma~$1$
\\\hline

$W$ &$W,X_1,x,u,v,y$, $Y_1,K,J_1$ & $W,X_1,x,u,v,y$, $Z,J_2$ &
$W,X_1,x,u,v,y$, $Z,J_3$ & $W,X_1,x,u,v,y$, $Z,J_{4,1}$ &
$W,X_1,x,u,v,y,Z$ and Lemma~$1$
\\\hline

$Z$ &$Z,x,u,v,y,Y_1,K$, $J_1$ & $Z,x,u,v,y,Z,J_2$ & $Z,x,u,v,y,Z,$
$J_3$ & $Z,x,u,v,y,Z,$ $J_{4,1}$ & $Z,x,u,v,y,Z$ and Lemma~$1$
\\\hline

$I$ &$I,X_1,x,u,v,y$, $Y_1,K,J_1$ & $I,X_1,x,u,v,y$, $Z,J_2$ &
$I,X_1,x,u,v,y$, $Z,J_3$ & $I,X_1,x,u,v,y,$ $Z,J_{4,1}$ &
$I,X_1,x,u,v,y,Z$ and Lemma~$1$
\\\hline

$X_1$ &$X_1,x,u,v,y,Y_1$, $K$, $J_1$ & $X_1,x,u,v,y,Z,$ $J_2$ &
$X_1,x,u,v,y,Z,$ $J_3$ & $X_1,x,u,v,y,Z,$ $J_{4,1}$ &
$X_1,x,u,v,y,Z$ and Lemma~$1$
\\\hline

$X_2$ &$X_2,X_1,x,u,v,y$, $Y_1,$ $K,J_1$ & $X_2,X_1,x,u,v,y$, $Z,$
$J_2$ & $X_2,X_1,x,u,v,y$, $Z,$ $J_3$ & $X_2,X_1,x,u,v,y$, $Z,$
$J_{4,1}$ & $X_2,X_1,x,u,v,y$, $Z$ and Lemma~$1$
\\\hline

$X_3$ & Lemma~$1$ and $x,u,v,y,Y_1,K$, $J_1$ & Lemma~$1$ and
$x,u,v,y,Z,$ $J_2$ & Lemma~$1$ and $x,u,v,y,Z,$ $J_3$ & Lemma~$1$
and $x,u,v,y,Z,J_{4,1}$ & Lemma~$1$,$x,u,v,y,Z$ and Lemma~$1$
\\\hline
\end{tabular}
\end{center}}
\centerline{Table $1$. Directed paths in $\overrightarrow{G}.$}
\end{figure}
\begin{figure}[htbp]
{\scriptsize
\begin{center}
\begin{tabular}{|p{0.5cm}|p{2.1cm}|p{2cm}|p{2cm}|p{2cm}|p{2.5cm}|}
\hline & $J_1$ & $J_2$ & $J_3$ & $J_{4,1}$  & $J_{4,2}$\\\hline

$J_1$ &$J_1,Z,x,u,v,y,Y_1,$ $K,J_1$& $J_1,Z,x,u,v,y,$ $Z,J_2$ &
$J_1,Z,x,u,v,y,$ $Z,J_3$ & $J_1,Z,x,u,v,y,$ $Z,J_{4,1}$ &
$J_1,Z,x,u,v,$ $y,$ $Z$ and Lemma $1$
\\\hline

$J_2$ &$J_2,I,X_1,x,u,v$, $y,Y,K,J_1$& $J_2,I,X_1,x,u,v,$ $y,Z,J_2$
& $J_2,I,X_1,x,u,v$, $y,Z,J_3$ & $J_2,I,X_1,x,u,v,$ $y,Z,J_{4,1}$ &
$J_2,I,X_1,x,u,v,y,Z$ and Lemma~$1$
\\\hline

$J_3$ &$J_3,W,X_1,x,u,v$, $y,Y,K,J_1$& $J_3,W,X_1,x,u,v,$ $y,Z,J_2$
& $J_3,W,X_1,x,u,v$, $y,Z,J_3$ & $J_3,W,X_1,x,u,v,$ $y,Z,J_{4,1}$ &
$J_3,W,X_1,x,u,v,y$, $Z$ and Lemma~$1$
\\\hline

$J_{4,1}$ &$J_{4,1},Z,x,u,v,y$, $Y,K,J_1$& $J_{4,1},Z,x,u,v,y,$
$Z,J_2$ & $J_{4,1},Z,x,u,v,y,$ $Z,J_3$ & $J_{4,1},Z,x,u,v,y,$
$Z,J_{4,1}$ & $J_{4,1},Z,x,u,v,$ $y,$ $Z$ and Lemma $1$
\\\hline

$J_{4,2}$ & Lemma~$1$ and $Z,x,u,v,y,Y_1,K$, $J_1$& Lemma~$1$ and
$Z,x,u,v,y,Z$, $J_2$ & Lemma~$1$ and $Z,x,u,v,y,Z$, $J_3$ &
Lemma~$1$ and $Z,x,u,v,y,Z$, $J_{4,1}$ & Lemma~$1$ and $Z$,
$x,u,v,y,Z$ and Lemma $1$
\\\hline
\end{tabular}
\end{center}}
\centerline{Table $2$. Directed paths in $\overrightarrow{G}$}
\end{figure}

If $(s,t)\in (V(G)\setminus J,J)$, then a path with length at most
$9$ from $s$ to $t$ is presented in Table~$1$. The notation
$X_1,X_2,\cdots,X_k$ means that there exists some desired path
$P=(x_1,x_2,\ldots,x_k)$, where $x_i\in X_i,\ i=1,2,\ldots,k$.

Note that, in Figure~$3$, at least two of $X_1,Y_1,Z$ are nonempty;
some sets may be empty (if so, the proof becomes easier). Though we
have not indicated all kinds of cases, in fact, we can get
information from these sequences of Table~1. For example, if
$(s,t)\in (X_3,J_3)$, then $X_3$ and $J_3$ are nonempty. By the
definition of set $J_3$, we know that $W,X_1$ and $Y_1$ are also
nonempty. In Table~$1$, $(s,t)\in (X_3,J_3)$ corresponds to
``Lemma~$1$ and $x,u,v,y,Z,$ $J_3$''. ``Lemma~$1$'' means that there
exists a path with length at most $2$ from $s$ to $x$ by the
definition of the above orientation and Lemma~1.

If  $(s,t)\in (J,J)$, then a path with length at most $9$ from $s$
to $t$ is presented in Table~$2$. If $(s,t)\in (J,V(G)\setminus J)$,
then a path with length at most $9$ from $s$ to $t$ is presented in
Table~$3$.

\begin{figure}[htbp]
{\scriptsize
\begin{center}
\begin{tabular}{|p{0.5cm}|p{1.1cm}|p{1.1cm}|p{1.1cm}|p{1.1cm}|p{1.6cm}|p{1.6cm}|p{1.8cm}|}
\hline & $x$ & $u$ & $v$ & $y$  & $Y_1$ & $Y_2$ & $Y_3$ \\\hline

$J_1$ &$J_1,Z$, $x$& $J_1,Z,x$, $u$ & $J_1,Z,x$, $u,v$ & $J_1,Z,x$,
$u,v,y$ & $J_1,Z,x,u,v,$ $y,$ $Y_1$ & $J_1,Z,x,u,v,$ $y,$ $Y_1,Y_2$
& $J_1,Z,x,u,v,y,$ and Lemma~$1$
\\\hline

$J_2$ &$J_2,I,X_1$, $x$& $J_2,I,X$, $x,u$ & $J_2,I,X_1,$ $x,u,v$ &
$J_2,I,X_1$, $x,u,v,y$ & $J_2,I,X_1,x,u$, $v,y,Y_1$ & $J_2,I,X_1,x$,
$u,v,y,Y_1,Y_2$ & $J_2,I,X_1,x,u$, $v$, $y,$ and Lemma~$1$
\\\hline

$J_3$ &$J_3,W,X_1$, $x$& $J_3,W,X_1$, $x,u$ & $J_3,W,X_1$, $x,u,v$ &
$J_3,W,X_1$, $x,u,v,y$ & $J_3,W,X_1,x$, $u,v,y,Y_1$ & $J_3,W,X_1,x$,
$u,v,y,Y_1,Y_2$ & $J_3,W,X_1,x$, $u,v,y,$ and Lemma~$1$
\\\hline

$J_{4,1}$ &$J_{4,1},Z$, $x$& $J_{4,1},Z$, $x,u$ & $J_{4,1},Z$,
$x,u,v$ & $J_{4,1},Z$, $x,u,v,y$ & $J_{4,1},Z,x,u$, $v,y,Y_1$ &
$J_{4,1},Z,x,u$, $v,y,Y_1,Y_2$ & $J_{4,1},Z,x,u,v,$
$y,$~and~Lemma~$1$
\\\hline

$J_{4,2}$ & Lemma~1 and $Z$, $x$& Lemma~1 and~$Z,x$, $u$ & Lemma~1
and~$Z,x$, $u,v$ & Lemma~1 and~$Z,x$, $u,v,y$ & Lemma~1~and $Z,
x,u,v,y,Y_1$ & Lemma~1~and $Z,x,u,v,y$, $Y_1,Y_2$ & Lemma~1~$Z$,
$x,u,v,y,$ and Lemma~$1$
\\\hline
\end{tabular}
\vspace*{10pt}

\centerline{\normalsize Table $3$. Directed paths in
$\overrightarrow{G}.$}
\end{center}}
\end{figure}

\begin{figure}[htbp]
{\scriptsize
\begin{center}
\begin{tabular}{|p{0.5cm}|p{1.1cm}|p{1.1cm}|p{1.1cm}|p{1.1cm}|p{1.9cm}|p{1.4cm}|p{1.8cm}|}
\hline & $K$ & $W$ & $Z$ & $I$  & $X_1$ & $X_2$ & $X_3$\\\hline

$J_1$ & $J_1,Z,x$, $u,v,y,$ $Y,$ $K$ & $J_1,Z,x$, $u,v,y,$ $Y,$ $W$
& $J_1,Z,x$, $u,v,y,$ $Z$ & $J_1,Z,x$, $u,v,y,$ $Z,I$ &
$J_1,Z,x,u,v,y,$ and Obs~$3$ & $J_1,Z,x,X_2$ & $J_1,Z,x,$ and
Lemma~$1$
\\\hline

$J_2$ & $J_2,I,X_1$, $x,u,v,y,$ $Y_1,$ $K$ & $J_2,I,X_1$, $x,u,v,y,$
$Y_1,$ $W$ & $J_2,I,X_1$, $x,u,v,y,$ $Z$ & $J_2,I,X$, $x,u,v,y,$
$Z,I$ & $J_2,I,X_1,x,u,v$, $y,$~and~Obs~$3$ & $J_2,I,X_1,x$, $X_2$&
$J_2,I,X_1,x$ and Lemma~$1$
\\\hline

$J_3$ & $J_3,W,X$, $x,u,v,$ $y,$ $Y,$ $K$ & $J_3,W,X$, $x,u,v,$ $y,$
$Y,$ $W$ & $J_3,W,X$, $x,u,v,$ $y,$ $Z$ & $J_3,W,X$, $x,u,v,$ $y,$
$Z,I$ & $J_3,W,X$, $x,u,v,$ $y,$~and~Obs~$3$ & $J_3,W,X,x$,  $X_2$ &
$J_3,W,X,x,$ and Lemma~$1$
\\\hline

$J_{4,1}$ & $J_{4,1},Z,x$, $u,v,y,$ $Y_1,$ $K$ & $J_{4,1},Z,x$,
$u,v,y,$ $Y_1,$ $W$ & $J_{4,1},Z,x$, $u,v,y,$ $Z$ & $J_{4,1},Z,x$,
$u,v,y,$ $Z,I$ & $J_{4,1},Z,x,u,v,y,$ and Obs~$3$ & $J_{4,1},Z,x$,
$X_2$ & $J_{4,1},Z,x,$ and Lemma~$1$
\\\hline

$J_{4,2}$ & Lemma~$1$, $Z,x,u$, $v,y,Y,$ $K$ & Lemma~$1$ $Z,x,u$,
$v,y,Y,$ $W$ & Lemma~$1$ $Z,x,u$, $v,y,Z$ & Lemma~$1$ $Z,x,u$,
$v,y,Z$,$I$ & Lemma~$1$, $Z$, $x,u,v,y,$ and Obs~$3$ & Lemma~$1$,
$Z$, $x,X_2$ & Lemma~$1\ Z$, $x,$~and~Lemma~$1$
\\\hline
\end{tabular}
\vspace*{10pt} \centerline{\normalsize Table $4$. Directed paths in
$\overrightarrow{G}.$}
\end{center}}
\end{figure}

{\bf Case $2.$} $Z= \emptyset$. By definition of the above set
partition, we know that $I=J=K=\emptyset$, and this case can be
checked easily. Observation~$4$ is used in this case.
\end{proof}

\begin{lemma}
Let $G\in \mathrm{Min} G(n, 3, \lambda, 1)$, where $n\geq 5$ and
$\lambda\geq 4$. If $G$ does not possess two adjacent vertices of
degree $2$, then the oriented diameter of $G$ is at most $9$.
\end{lemma}

\begin{proof}
Let $G\in \mathrm{Min} G(n, 3, \lambda, 1)$ which does not possess
two adjacent vertices of degree $2$, where $n\geq 5$ and
$\lambda\geq 4$. By Lemma $2$, $G$ must be one of the graphs
$H^8,H^9,H^{10}_j (j=1,2,3,4)$ and $H^{12}_j(j=1,2,3)$ if Figure
$3$. It suffices to show that the oriented diameters of these graphs
is at most $9$. In fact, we can orient theses graphs as Figure $4$.
It is easy to check that the diameters of
$\overrightarrow{H^8},\overrightarrow{H^9},\overrightarrow{H^{10}_j}
(j=1,2,3,4)$ and $\overrightarrow{H^{12}_j}(j=1,2,3)$ are at most
$9$. Thus we are done.
\end{proof}

\begin{figure}[h,t,b]
\begin{center}
\scalebox{0.7}[0.7]{\includegraphics{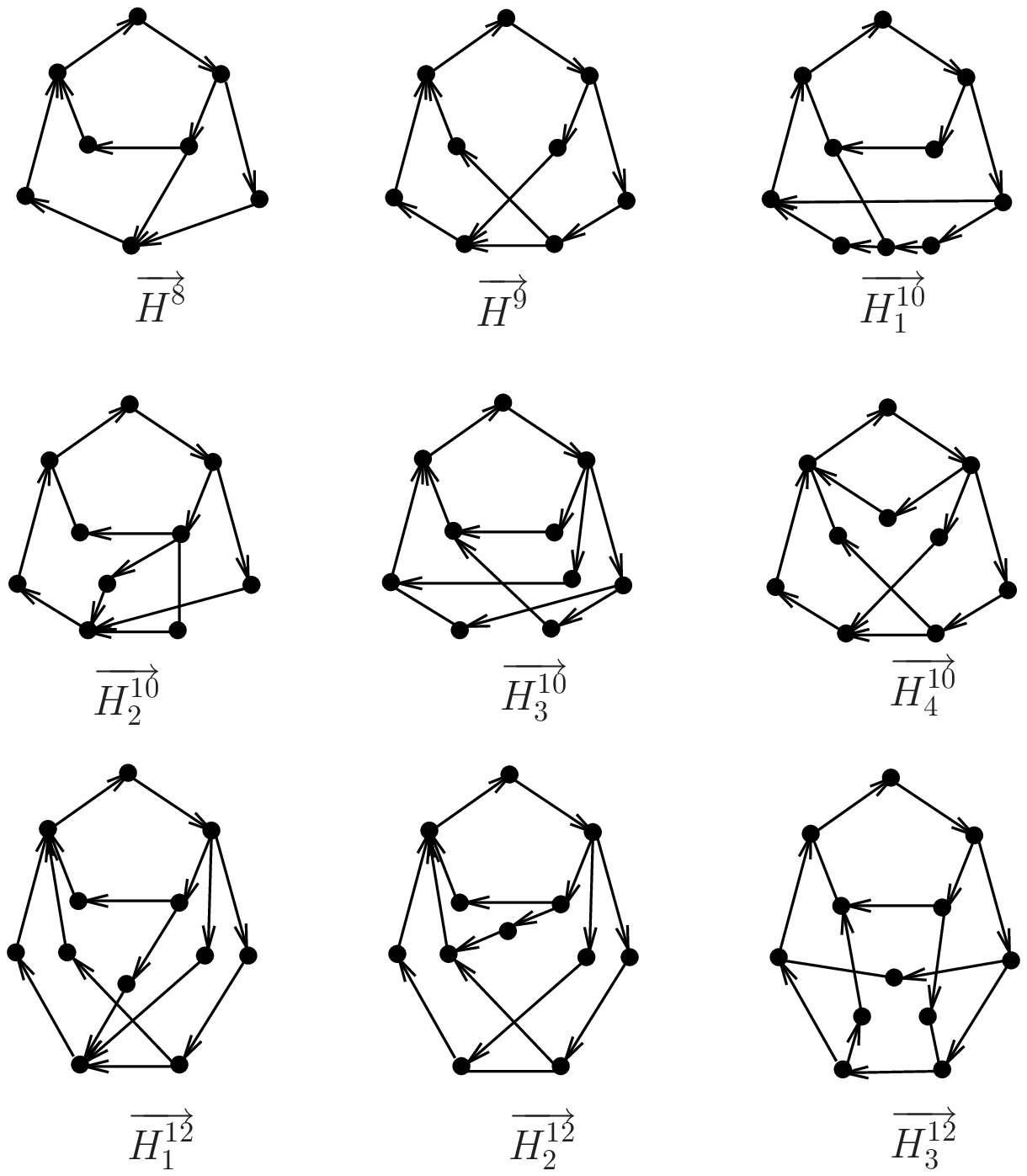}}

Figure 4. the orientations of $H^8,H^9,H^{10}_j (j=1,2,3,4)$ and
$H^{12}_j(j=1,2,3).$
\end{center}
\end{figure}

Combining Lemmas $3$ and $4$, we know that the following theorem
holds.
\begin{theorem}
Let $G\in \mathrm{Min} G(n, 3, \lambda, 1)$, where $n\geq 5$ and
$\lambda\geq 4$. Then the oriented diameter of $G$ is at most $9$.
\end{theorem}

Since the oriented diameter of $G$ is no larger than the oriented
diameter of its spanning subgraph. We have the following corollary.

\begin{corollary}
If a graph $G$ has a spanning bridgeless subgraph with diameter at
most $3$, which admits two adjacent vertices of degree $2$, then the
oriented diameter of $G$ is at most $9$.
\end{corollary}

\begin{corollary}
If a graph $G$ admits a spanning subgraph contained in $\mathrm{Min}
G(n, 3, \lambda, 1)$, then the oriented diameter of $G$ is at most
$9$.\end{corollary}

By the above arguments, we propose the following conjecture.

\begin{conjecture}
Every bridgeless graphs with diameter at most $3$ has an oriented
diameter at most $9$.
\end{conjecture}

If the above conjecture is true, then by Proposition $1$, one can
get that $f(3)=9.$

\end{document}